\documentclass[11pt,a4paper]{amsart}
\usepackage{amssymb}
\usepackage{mathabx}
\usepackage{mathrsfs}
\usepackage{syntonly}
\usepackage{amsmath}
\usepackage{amsthm}
\usepackage{amsfonts}
\usepackage{amssymb}
\usepackage{latexsym}
\usepackage{amscd,amssymb,amsopn,amsmath,amsthm,graphics,amsfonts,mathrsfs,accents,enumerate,verbatim,calc}
\usepackage[dvips]{graphicx}
\usepackage[colorlinks=true,linkcolor=red,citecolor=blue]{hyperref}
\usepackage[all]{xy}

\date{}
\allowdisplaybreaks[4] \footskip=15pt
\renewcommand{\uppercasenonmath}[1]{}

\numberwithin{equation}{section} \theoremstyle{plain}
\newtheorem{lem}{Lemma}[section]
\newtheorem{cor}[lem]{Corollary}
\newtheorem{prop}[lem]{Proposition}
\newtheorem{thm}[lem]{Theorem}

\newtheorem{definition}[lem]{Definition}
\newtheorem{Ex}[lem]{Example}
\newtheorem{Quest}[lem]{Question}
\newtheorem{Property}[lem]{Property}
\newtheorem{Properties}[lem]{Properties}
\newtheorem{Subprops}{}[lem]
\newtheorem{Para}[lem]{}

\newtheorem{rem}[lem]{Remark}

\newenvironment{df}{\begin{definition}\rm}{\end{definition}}
\newenvironment{ex}{\begin{Ex}\rm}{\end{Ex}}

\newtheorem*{ack*}{ACKNOWLEDGEMENTS}

\newcommand{\pf}{\noindent\begin {proof}}
\newcommand{\epf}{\end{proof}}

\newcommand{\X}{\mathcal{X}}
\newcommand{\W}{\mathcal{W}}

\newcommand{\C}{\mathcal{C}}

\newcommand{\E}{\mathbb{E}}

\begin{document}
\begin{center}
{\large  \bf  A new method to construct model structures from
left Frobenius pairs in extriangulated categories}

\vspace{0.5cm}  Yajun Ma, Haiyu Liu, Yuxian Geng\footnote{Corresponding author. }

\end{center}

\bigskip
\centerline { \bf  Abstract}
\medskip

\leftskip10truemm \rightskip10truemm \noindent Extriangulated categories were introduced by Nakaoka and Palu as a simultaneous generalization of exact categories and triangulated categories. In this paper, we first introduce the concept of left Frobenius pairs on an extriangulated category $\C$, and then establish a bijective correspondence between Frobenius pairs and certain cotorsion pairs in $\C$.  As an application, some new admissible model structures are established from left Frobenius pairs under certain conditions, which generalizes a result of Hu et al. (J. Algebra 551 (2020) 23-60).    \\[2mm]

\leftskip0truemm \rightskip0truemm
\section { \bf Introduction}

The notion of extriangulated categories, whose extriangulated structures are given by $\E$-triangles with some axioms, was introduced by Nakaoka and
Palu in \cite{NP} as a simultaneous generalization of exact categories and triangulated categories. They gave a bijective correspondence between Hovey twin cotorsion pairs and admissible model structures which unified the work of Hovey, Gillespie and Yang (see \cite{HCc,Gillespie,Yang}). Exact categories and extension closed subcategories of an
extriangulated category are extriangulated categories, while there exist some other examples of extriangulated categories which are neither exact nor triangulated, see \cite{NP,ZZ,HZZ}.

Motivated by the ideas of projective covers and injective envelopes, Auslander and Buchweitz analyzed the framework in which the theory of maximal Cohen-Macaulay approximation can be developed. They systematically established their theory in abelian categories, which is known as Auslander-Buchweitz approximation theory. Up to now, Auslander-Buchweitz approximation theory has many important applications, see for example \cite{MSSS1,MSSS2,BMPS,DLWW}.
In particular, Becerril and coauthers \cite{BMPS}  have revisited
Auslander-Buchweitz approximation Theory. From the notions of relative generators and cogenerators in approximation theory, they introduced the concept of left Frobenius pairs in an abelian category, established a bijective correspondence between Frobenius pairs and relative cotorsion pairs, and showed how to construct an exact model structure from a Frobenius
pair, as a result of Hovey-Gillespie correspondence applied to two complete cotorsion pairs on
an exact category (see \cite{HCc,Gillespie}).

The aim of this paper is to introduce the concept of left Frobenius pairs in an extriangulated category and give a method to construct more admissible model structures from left Frobenius pairs. For this purpose, we need establish a bijective correspondence between left Frobenius pairs and cotorsion pairs in an extriangulated category under certain conditions.

The paper is organized as follows. In Section 2, we recall the definition of an extriangulated category and outline some basic properties that will be used later. In Section 3, we first introduce the concept of left Frobenius pairs (see Definition \ref{df:Frobenius-pair}), and then study relative resolution dimension and thick subcategories with respect to a given left Frobenius pair. As a result, we give a bijective correspondence between left Frobenius pairs and cotorsion pairs in an extriangulated category under certain conditions (see Theorem \ref{thm}).
In Section 4, we give a method to construct the admissible model structure from a strong left Frobenius pair under certain conditions (see Theorem \ref{thm6}), which generalizes a main result of Hu et al. in \cite{HZZ}. This is based on the bijective correspondence established in Section 3.

\section{\bf Preliminaries}
Throughout this paper, $\C$ denotes an additive category, by the term $``subcategory"$ we always mean a full additive subcategory of an additive category closed under isomorphisms and direct summands.
 We denote by ${\mathcal{\C}}(A, B)$ the set of morphisms from $A$ to $B$ in $\C$.

Let $\X$ and $\mathcal{Y}$ be two subcategories of $\C$, a morphism $f: X\rightarrow C$ in $\C$ is said to be an $\X$-precover of $C$ if $X\in\X$ and ${\C}(X', f): {\C}(X', X)\rightarrow {\C}(X', C)$ is surjective for all $X'\in\X$. If any $C\in \mathcal{Y}$ admits an $\X$-precover, then $\X$ is called a precovering class in $\mathcal{Y}$.
 By dualizing the definitions above, we get notions of an $\X$-preenvelope of $C$ and a preenveloping class in $\mathcal{Y}$. For more details, we refer to \cite{EJ2}.

Let us briefly recall some definitions and basic properties of extriangulated categories from \cite{NP}. We omit some details here, but the reader can find them in \cite{NP}.

Assume that $\mathbb{E}: \mathcal{C}^{\rm op}\times \mathcal{C}\rightarrow {\rm Ab}$ is an additive bifunctor, where $\mathcal{C}$ is an additive category and ${\rm Ab}$ is the category of abelian groups. For any objects $A, C\in\mathcal{C}$, an element $\delta\in \mathbb{E}(C,A)$ is called an $\mathbb{E}$-extension.
Let $\mathfrak{s}$ be a correspondence which associates an equivalence class $\mathfrak{s}(\delta)=\xymatrix@C=0.8cm{[A\ar[r]^x
 &B\ar[r]^y&C]}$ to any $\mathbb{E}$-extension $\delta\in\mathbb{E}(C, A)$. This $\mathfrak{s}$ is called a {\it realization} of $\mathbb{E}$, if it makes the diagram in \cite[Definition 2.9]{NP} commutative.
 A triplet $(\mathcal{C}, \mathbb{E}, \mathfrak{s})$ is called an {\it extriangulated category} if it satisfies the following conditions.
\begin{enumerate}
\item $\mathbb{E}\colon\mathcal{C}^{\rm op}\times \mathcal{C}\rightarrow \rm{Ab}$ is an additive bifunctor.

\item $\mathfrak{s}$ is an additive realization of $\mathbb{E}$.

\item $\mathbb{E}$ and $\mathfrak{s}$ satisfy certain axioms in \cite[Definition 2.12]{NP}.
\end{enumerate}

In particular, we recall the following axioms which will be used later:

\emph{(ET4)} Let $\delta\in\mathbb{E}(D,A)$ and $\delta'\in\mathbb{E}(F, B)$ be $\mathbb{E}$-extensions realized by
 \begin{center} $\xymatrix{A\ar[r]^f&B\ar[r]^{f'}&D}$ and $\xymatrix{B\ar[r]^g&C\ar[r]^{g'}&F}$\end{center}
 respectively. Then there exists an object $E\in\mathcal{C}$, a commutative diagram
 $$\xymatrix{A\ar[r]^f\ar@{=}[d]&B\ar[r]^{f'}\ar[d]_g&D\ar[d]^d\\
A\ar[r]^h&C\ar[r]^{h'}\ar[d]_{g'}&E\ar[d]^e\\
&F\ar@{=}[r]&F}$$
in $\mathcal{C}$, and an $\mathbb{E}$-extension $\delta^{''}\in\mathbb{E}(E, A)$ realized by $\xymatrix{A\ar[r]^h&C\ar[r]^{h'}&E,}$
which satisfy the following compatibilities.

\begin{enumerate}
\item[(i)] $\xymatrix{D\ar[r]^d&E\ar[r]^{e}&F}$ realizes $f'_*\delta'$,

\item[(ii)] $d^*\delta^{''}=\delta$,

\item[(iii)] $f_*\delta^{''}=e^*\delta'$.
\end{enumerate}

\emph{(ET4)$^{\rm op}$} Dual of (ET4).

\begin{rem}
Note that both exact categories and triangulated categories are extriangulated categories $($see \cite[Example 2.13]{NP}$)$ and extension closed subcategories of extriangulated categories are
again extriangulated $($see \cite[Remark 2.18]{NP}$)$. Moreover, there exist extriangulated categories which
are neither exact categories nor triangulated categories $($see \cite[Proposition 3.30]{NP}, \cite[Example 4.14]{ZZ} and \cite[Remark 3.3]{HZZ}$)$.
\end{rem}

%
%
\begin{lem}{\rm \cite[Corollary 3.12]{NP}}  Let $(\mathcal{C}, \mathbb{E}, \mathfrak{s})$ be an extriangulated category and $$\xymatrix@C=2em{A\ar[r]^{x}&B\ar[r]^{y}&C\ar@{-->}[r]^{\delta}&}$$ an $\mathbb{E}$-triangle. Then we have the following long exact sequences:

$\xymatrix@C=1cm{\mathcal{C}(C, -)\ar[r]^{\mathcal{C}(y, -)}&\mathcal{C}(B, -)\ar[r]^{\mathcal{C}(x, -)}&\mathcal{C}(A, -)\ar[r]^{\delta^\sharp}&\mathbb{E}(C, -)\ar[r]^{\mathbb{E}(y, -)}&\mathbb{E}(B, -)\ar[r]^{\mathbb{E}(x, -)}&\mathbb{E}(A, -);}$

$\xymatrix@C=1cm{\mathcal{C}(-, A)\ar[r]^{\mathcal{C}(-, x)}&\mathcal{C}(-, B)\ar[r]^{\mathcal{C}(-, y)}&\mathcal{C}(-, C)\ar[r]^{\delta_\sharp}&\mathbb{E}(-, A)\ar[r]^{\mathbb{E}(-, x)}&\mathbb{E}(-, B)\ar[r]^{\mathbb{E}(-, y)}&\mathbb{E}(-, C),}$

\noindent where natural transformations $\delta_\sharp$ and $\delta^\sharp$ are induced by $\mathbb{E}$-extension $\delta\in \mathbb{E}(C, A)$ via Yoneda's lemma.

\end{lem}

Let $\mathcal{C},\mathbb{E}$ be as above.

$\bullet$ A sequence $\xymatrix@C=0.6cm{A \ar[r]^{x} & B \ar[r]^{y} & C }$ is called a \emph{conflation} if it realizes some $\mathbb{E}$-extension $\delta\in\mathbb{E}(C, A)$.
In this case, $x$ is called an {\it inflation}, $y$ is called a {\it deflation}, and we write it as $$\xymatrix{A\ar[r]^x&B\ar[r]^{y}&C\ar@{-->}[r]^{\delta}&.}$$
We usually do not write this ``$\delta$" if it is not used in the argument.

$\bullet$ An object $P\in \mathcal{C}$ is called projective if for any $\E$-triangle $\xymatrix{A\ar[r]^{x}&B\ar[r]^{y}&C\ar@{-->}[r]^{\delta}&}$ and any morphism $c\in \mathcal{C}(P, C)$, there exists $b\in \mathcal{C}(P, B)$ satisfying $y\circ b=c$. Injective objects are defined dually. We denote the subcategory consisting of projective (resp., injective)~objects in $\C$ by $Proj(\C)$~(resp., $Inj(\C)$).

$\bullet$ We say $\mathcal{C}$ has enough projectives (resp., enough injectives) if for any object $C\in\C$~(resp., $A\in \mathcal{C}$), there exists an $\E$-triangle
$\xymatrix{A\ar[r]^{x}&P\ar[r]^{y}&C\ar@{-->}[r]^{\delta}&}$ ($\xymatrix{A\ar[r]^{x}&I\ar[r]^{y}&C\ar@{-->}[r]^{\delta}&})$ satisfying $P\in Proj(\mathcal{C})$~$(\mathrm{resp}., I\in Inj(\mathcal{C})$). In this case, $A$ is called the syzygy of $C$ (resp., $C$ is called the cosyzygy of $A$) and is denoted by $\Omega(C)$~$(\mathrm{resp}.,\Sigma(A)$).

\begin{rem} \rm{(1) If $(\mathcal{C}, \mathbb{E}, \mathfrak{s})$ is an exact category, then the definitions of enough projectives and enough injectives agree with the usual definitions.

(2) If $(\mathcal{C}, \mathbb{E}, \mathfrak{s})$ is a triangulated category, then $Proj(\mathcal{C})$ and $Inj(\mathcal{C})$ consist of zero objects.
}\end{rem}

For a subcategory $\mathcal{B}\subseteq \mathcal{C}$, put $\Omega^{0}\mathcal{B}=\mathcal{B}$, and for $i>0$, we define $\Omega^{i}\mathcal{B}$ inductively to be the subcategory consisting of syzygies of objects in $\Omega^{i-1}\mathcal{B}$, i.e., $\Omega^{i}\mathcal{B}=\Omega(\Omega^{i-1}\mathcal{B}).$
We call $\Omega^{i}\mathcal{B}$ the $i$-th syzygy of $\mathcal{B}$. Dually we define the $i$-th cosyzygy $\Sigma^{i}\mathcal{B}$ by $\Sigma^{0}\mathcal{B}=\mathcal{B}$ and $\Sigma^{i}\mathcal{B}=\Sigma(\Sigma^{i-1}\mathcal{B})$ for $i>0.$

In \cite{LN} the authors defined higher extension groups in an extriangulated category having enough projectives and injectives as $\E^{i+1}(X, Y)\cong\E(X, \Sigma^{i}Y)\cong\E(\Omega^{i}X, Y)$ for $i\geq 0$, and they showed the following result:

\begin{lem} \cite[Proposition 5.2]{LN} Let $\xymatrix{A\ar[r]^{x}&B\ar[r]^{y}&C\ar@{-->}[r]^{\delta}&}$ be an $\E$-triangle. For any object $X\in \mathcal{B}$, there are long exact sequences
$$\xymatrix@C=0.5cm{\cdots\ar[r] &\E^{i}(X, A)\ar[r]^{x_{*}}&\E^{i}(X, B)\ar[r]^{y_{*}}&\E^{i}(X, C)\ar[r]&\E^{i+1}(X, A)\ar[r]^{x_{*}}&\E^{i+1}(X, B)\ar[r]^{y_{*}}&\cdots(i\geq1),}$$
$$\xymatrix@C=0.5cm{\cdots\ar[r] &\E^{i}(C, X)\ar[r]^{y^{*}}&\E^{i}(B, X)\ar[r]^{x^{*}}&\E^{i}(A, X)\ar[r]&\E^{i+1}(C, X)\ar[r]^{y^{*}}&\E^{i+1}(B, X)\ar[r]^{x^{*}}&\cdots(i\geq1).}$$
\end{lem}

An $\E$-triangle sequence in $\mathcal{C}$ \cite{ZZ1} is displayed as a sequence
$$\xymatrix{\cdots\ar[r]&X_{n+1}\ar[r]^{d_{n+1}}&X_{n}\ar[r]^{d_{n}}&X_{n-1}\ar[r]&\cdots&}$$
over $\mathcal{C}$ such that for any $n$, there are $\E$-triangles $\xymatrix{K_{n+1}\ar[r]^{g_{n}}&X_{n}\ar[r]^{f_{n}}&K_{n}\ar@{-->}[r]^{\delta^{n}}&}$ and the differential $d_{n}=g_{n-1}f_{n}$.

From now on to the end of the paper, we always suppose that $(\mathcal{C}, \mathbb{E}, \mathfrak{s})$ is an extriangulated categories with enough projectives and injectives.

\section{\bf Frobenius pairs and cotorsion pairs}
In this section, we introduce the concept of Frobenius pairs and show that it has very nice homological properties, which are necessary to construct cotorsion pairs from Frobenius pairs. At first, we need introduce the following definitions.

\begin{df} Let $\mathcal{X}$ be a subcategory of $\mathcal{C}$.
\begin{enumerate}
\item For any non-negative integer $n$, we denote by $\widecheck{\X_{n}}$ (resp., $\widehat{\X_{n}}$) the class of objects $C\in \C$ such that there exists an $\E$-triangle sequence ~\begin{center}$C\rightarrow X_{0}\rightarrow\cdots\rightarrow X_{n-1}\rightarrow X_{n}$ (resp., $X_{n}\rightarrow X_{n-1}\rightarrow\cdots\rightarrow X_{0}\rightarrow C$)
    \end{center} with each $X_{i}\in \mathcal{X}$. Moreover, we set $~\widecheck{\X}=\bigcup\limits_{n=0}^\infty \widecheck{\X_{n}}$, $\widehat{\X}=\bigcup\limits_{n=0}^\infty \widehat{{\X_n}}$.

\item For any $C\in \mathcal{C}$, the \emph{$\mathcal{X}$-resolution dimension} of $C$ is defined as
\begin{center}
resdim$_{\mathcal{X}}(C)$:=min$\{n\in\mathbb{N}:C\in\widehat{\mathcal{X}_{n}}\}$.
\end{center}
If $C\not\in\widehat{\X_{n}}$ for any $n\in \mathbb{N}$, then resdim$_{\X}(C)=\infty$.
\end{enumerate}
\end{df}


For a subcategory $\mathcal{X}$ of $\mathcal{C}$, define $\mathcal{X}^{\perp}=\{Y\in\mathcal{C}|\E^{i}(X, Y)=0 \mathrm{~for~all~} i\geq 1, \mathrm{~and~all~} X \in \mathcal{X}\}$. Similarly, we can define $^{\perp}\mathcal{X}.$

\begin{df}
Let $\mathcal{X}$ and $\mathcal{W}$ be two subcategories of $\mathcal{C}$. We say that

(1) $\mathcal{W}$ is a \emph{cogenerator} for $\mathcal{X}$, if $\mathcal{W}\subseteq\mathcal{X}$ and for each object $X\in \mathcal{X}$, there exists an $\E$-triangle $\xymatrix@C=0.6cm{X\ar[r]^{}&W\ar[r]^{}&X'\ar@{-->}[r]^{\delta}&}$ with $W\in \W$ and $X^{'}\in \X$. The notion of a \emph{generator} is defined dually.

(2) $\mathcal{W}$ is \emph{$\mathcal{X}$-injective} if $\mathcal{W}\subseteq\mathcal{X}^{\perp}$. The notion of an \emph{$\mathcal{X}$-projective} subcategory is defined dually.

(3) $\mathcal{W}$ is an \emph{$\mathcal{X}$-injective cogenerator} for $\X$ if $\mathcal{W}$ is a cogenerator for $\mathcal{X}$ and $\mathcal{W}\subseteq\mathcal{X}^{\perp}$. The notion of an \emph{$\mathcal{X}$-projective generator} for $\X$ is defined dually.

(4) $\X$ is a \emph{thick subcategory} if it is closed under direct summand and for any $\E$-triangle $$\xymatrix{A\ar[r]^{x}&B\ar[r]^{y}&C\ar@{-->}[r]^{\delta}&}$$ in $\C$ and two of $A,B,C$ are in $\C$, then so is the third.
\end{df}


The following theorem unifies some results of \cite{AB} and \cite{MSSS1}. It shows that any object in $\widehat{\X}$ admits two $\E$-triangles: one giving rise to an $\X$-precover and the other to a $\widehat{\W}$-preenvelope.
\begin{thm}\label{thm2}
Suppose $\X$ is closed under extensions and $\W$ is a cogenerator for $\X$. Consider the following conditions:

(1) $C$ is in $\widehat{\X_n}.$

(2) There exists an $\E$-triangle
$\xymatrix{Y_{C}\ar[r]^{}&X_{C}\ar[r]^{\varphi_{C}}&C\ar@{-->}[r]^{\delta}&}$
with $X_{C}\in \X$ and $Y_{C}\in \widehat{\W_{n-1}}.$

(3) There exists an $\E$-triangle
$\xymatrix{C\ar[r]^{\psi^{C}}&Y^{C}\ar[r]^{}&X^{C}\ar@{-->}[r]^{\theta}&}$
with $X^{C}\in \X$ and $Y^{C}\in \widehat{\W_{n}}.$

Then, $(1)\Leftrightarrow(2)\Rightarrow(3)$. If $\X$ is also closed under CoCone of deflations, then $(3)\Rightarrow(2)$, and hence all three conditions are equivalent. If $\W$ is $\X$-injective, then $\varphi_{C}$ is an $\X$-precover of $C$ and $\psi^{C}$ is a $\widehat{\W}$-preenvelope of $C$.

\end{thm}
\begin{proof}
The proof is dual to that of \cite[Proposition 3.6]{MDZH}.
\end{proof}



\begin{df}\label{df:Frobenius-pair}
A pair $(\X,\W)$ is called a left \emph{Frobenius pair} in $\C$ if the following holds:

 $(1)$ $\X$ is closed under extensions and CoCone of deflations,

 $(2)$ $\mathcal{W}$ is an $\mathcal{X}$-injective cogenerator for $\X$.

 If in addition $\W$ is also an $\mathcal{X}$-projective generator for $\X$, then we say $(\X,\W)$ is a \emph{strong left Frobenius pair}.
\end{df}

\begin{ex}(1)\label{ex} Assume that $\C=R$-$\mathrm{Mod}$ is the category of left $R$-modules for a ring $R$. A left $R$-module $N$ is called \emph{Gorenstein projective} \cite{Holm,EJ2}~if there is an exact sequence of projective modules
$$\mathbf{P}= \cdots\rightarrow P_{1}\rightarrow P_{0}\rightarrow P^{0}\rightarrow P^{1}\rightarrow\cdots$$
with $N=\mathrm{Ker}(P_{0}\rightarrow P^{0})$ such that $\mathrm{Hom}_{R}(\mathbf{P},Q)$ is exact for any projective left $R$-module $Q$. Let $\mathcal{GP}(R)$ be the full subcategory of $R$-$\mathrm{Mod}$ consisting of all Gorenstein projective modules and $\mathcal{P}(R)$ the subcategory of  $R$-$\mathrm{Mod}$ consisting of all projective modules. Then $(\mathcal{GP}(R),\mathcal{P}(R))$ is a strong left Frobenius pair.

(2) Let $\C$ be a triangulated category with a proper class $\xi$ of triangles.
Asadollahi and Salarian \cite{AS} introduced and studied $\xi$-Gprojective and
$\xi$-Ginjective objects, and developed a relative homological algebra in $\C$.
Let $\mathcal{GP}(\xi)$ denotes the full subcategory of $\xi$-Gprojective objects and $\mathcal{P}(\xi)$ denotes the full subcategory of $\xi$-projective objects.
Then $(\mathcal{GP}(\xi),\mathcal{P}(\xi))$ is a strong left Frobenius pair.

(3) Let $\mathcal{T}$ be a triangulated category, and let $\mathcal{M}$ be a silting subcategory of $\mathcal{T}$ with $\mathcal{M}=add\mathcal{M}$, where $add\mathcal{M}$ is the  smallest full subcategory of $\mathcal{T}$ which
contains $\mathcal{M}$ and which is closed under taking isomorphisms, finite direct sums, and
direct summands. Then $(\mathcal{T}_{\geq0}, \mathcal{M})$ is a left Frobenius pair by \cite[Corollary 3.7]{MDZH} and \cite[Proposition 2.7]{IY}, where $\mathcal{T}_{\geq0}:=\bigcup\limits_{n\geq0} \mathcal{M}[-n]*\cdots\mathcal{M}[-1]*\mathcal{M}$.

(4) In \cite{CLY}, the authors showed that if $(\mathcal{X},\mathcal{Y})$ is a complete and hereditary cotorsion pair in an abelian category $\mathcal{A}$ and $\mathcal{Y}$ is closed under kernels of epimorphisms, then $(\mathcal{G}(\mathcal{X})\bigcap\mathcal{Y},\mathcal{X}\bigcap\mathcal{Y})$ is a strong left Frobenius pair, where $\mathcal{G}(\mathcal{X})$ is the class of objects $M$ in $\mathcal{A}$ satisfying that there exists an exact sequence
$$\mathbf{X}= \cdots\rightarrow X_{1}\rightarrow X_{0}\rightarrow X^{0}\rightarrow X^{1}\rightarrow\cdots$$
with each term in $\mathcal{X}$ such that $M\cong\mathrm{Ker}(X^{0}\rightarrow X^{1})$ and ${\rm Hom}_{\mathcal{A}}(\mathbf{X},Q)$ is exact for any object $Q$ in $\mathcal{X}\bigcap\mathcal{Y}$.
\end{ex}
\begin{lem}\label{lem1}
Let $(\X,\W)$ be a left Frobenius pair. Given an $\E$-triangle
$\xymatrix{K\ar[r]^{x}&X\ar[r]^{y}&C\ar@{-->}[r]^{\delta}&}$ with $X\in\X$. Then $C\in \widehat{\X}$ if and only if $K\in \widehat{\X}$.
\end{lem}
\begin{proof}
The proof is dual to that of \cite[Lemma 3.8]{MDZH}.
\end{proof}

\begin{prop}\label{thm3}
Let $(\X,\W)$ be a left Frobenius pair. The following statements are equivalent for any $C\in \widehat{\X}$ and non-negative integer $n$.

(1) resdim$_{\X}(C)\leq n$.

(2) If $U\rightarrow X_{n-1}\rightarrow\cdots\rightarrow X_{0}\rightarrow C$ is an $\E$-triangle sequence with $X_{i}\in \X$ for $0\leq i\leq n-1$, then $U\in \X$.

\end{prop}
\begin{proof}
 $(2)\Rightarrow(1)$ is trivial.

$(1)\Rightarrow(2).$ Let $C$ be in $\widehat{\X}$. Then by Theorem \ref{thm2}, we have an $\E$-triangle sequence
$W_{n}\rightarrow \cdots \rightarrow W_{1}\rightarrow X  \rightarrow C$ with $X\in \X$ and $W_{i}\in\W$ for $1\leq i\leq n$.
Since $\mathcal{W}\subseteq\mathcal{X}^{\perp}$, it is easy to see that $\widehat{\mathcal{W}}\subseteq\mathcal{X}^{\perp}$.
Thus we have $\E^{n+i}(C,Y)\cong \E^{i}(W_{n},Y)=0$ for all $i\geq 1$ and $Y\in \widehat{\W}$.
 If $U\rightarrow X_{n-1}\rightarrow\cdots\rightarrow X_{0}\rightarrow C$ is an $\E$-triangle sequence with $X_{i}\in \X$ for $0\leq i\leq n-1$, then we have  $\E^{i}(U, Y)\cong\E^{n+i}(C,Y)=0$ for all $i\geq 1$ and $Y\in \widehat{\W}$. Note that $U\in \widehat{\X}$ by Lemma \ref{lem1}.
 Hence there exists an $\E$-triangle
$\xymatrix{Y_{U}\ar[r]^{}&X_{U}\ar[r]^{}&U\ar@{-->}[r]^{}&}$ with $X_{U}\in \X$ and $Y_{U}\in \widehat{\W}$ by Theorem \ref{thm2}. It follows that the above $\E$-triangle splits. Hence $U\in \X$.
\end{proof}
If $\X$ is a subcategory of $\C$, then we denote by $\mathrm{Thick}(\X)$ the smallest thick subcategory that contains $\X$. The following result shows that for a left Frobenius pair $(\X,\W)$, $\widehat{\X}$ is an extriangulated category. In particular, if $\C$ is a triangulated category, then $\widehat{\X}$ is the smallest triangulated subcategory of $\C$ containing $\X$ and is closed under direct summands and isomorphisms.
\begin{prop}\label{cor2}
Let $(\X,\W)$ be a left Frobenius pair. Then $\mathrm{Thick(\X)}=\widehat{\X}$.
\end{prop}
\begin{proof}
For any $\E$-triangle $\xymatrix{A\ar[r]^{x}&B\ar[r]^{y}&C\ar@{-->}[r]^{\delta}&,}$~we need to check that if any two of $A, B$ and $C$ are in $\widehat{\X}$, then the third one is in $\widehat{\X}$.
Since $\widehat{\X}$ is closed under extensions by the dual of \cite[Proposition 3.6]{MDZH}, it suffices to show that if $B\in\widehat{\X}$, then $A\in \widehat{\X}$ if and only if $C\in \widehat{\X}$. We first show that if $A$ and $B$ are in $\widehat{\X}$, then $C\in \widehat{\X}$. Since $B\in \widehat{\X}$, we have an $\E$-triangle
$\xymatrix{Y_{B}\ar[r]^{}&X_{B}\ar[r]^{}&B\ar@{-->}[r]^{}&}$ with $X_{B}\in \X, Y_{B}\in\widehat{\W}$. By $(ET4)^{op}$, we obtain a commutative diagram
$$\scalebox{0.9}[0.9]{     \xymatrix{Y_{B}\ar[r]^{}\ar@{=}[d]&L\ar[r]^{}\ar[d]_{}&A\ar[d]^{}&\\
Y_{B}\ar[r]^{}&X_{B}\ar[r]^{}\ar[d]_{}&B\ar[d]^{}&\\
&C\ar@{=}[r]&C.}}$$
 It follows that $L\in \widehat{\X}$ as $A$ and $Y_{B}$ are $\in \widehat{\X}$. Therefore $C\in \widehat{\X}.$

Suppose now $B$ and $C$ are in $\widehat{\X}$. It follows from Lemma \ref{lem1} that $L\in\widehat{\X}$. Applying Lemma \ref{lem1} again to the $\E$-triangle
$\xymatrix{Y_{B}\ar[r]^{}&L\ar[r]^{}&A\ar@{-->}[r]^{}&,}$ one has that $A\in \widehat{\X}$.

Suppose $C_{1}\oplus C_{2}\in \widehat{\X}$. We proceed by induction on $n=$resdim$_{\X}(C_{1}\oplus C_{2})$. If $n=0$, then $C_{1}$ and $C_{2}$ are in $\X$.

Suppose $n>0$. There is an $\E$-triangle
$\xymatrix@C=2.4em{K\ar[r]^{x}&X\ar[r]^{\tiny\begin{bmatrix}y_{1}\\y_{2}\end{bmatrix}}&C_{1}\oplus C_{2}\ar@{-->}[r]^{\delta}&}$
with $X\in \X$ and resdim$_{\X}(K)=n-1$. By $(ET4)^{op}$, we obtain the following commutative diagram

$$\scalebox{0.9}[0.9]{\xymatrix@C=3em{
  K\ar[r]^{} \ar@{=}[d] &L_{2} \ar[r]^{} \ar[d]^{x_{2}}& C_{1} \ar[d]^{\tiny\begin{bmatrix}1\\0\end{bmatrix}} \ar@{-->}[r]^{} &  \\
  K\ar[r]^{x} & X\ar[r]^{\tiny\begin{bmatrix}y_{1}\\y_{2}\end{bmatrix}} \ar[d]^{y_{2}} & C_{1}\oplus C_{2}\ar@{-->}[r]^{\delta} \ar[d]^{\tiny\begin{bmatrix}0&1\end{bmatrix}} &\\
  & C_{2} \ar@{-->}[d]^{\delta_{2}} \ar@{=}[r] &C_{2}\ar@{-->}[d]^{0}&~ \\
  &&~~~&}}$$
  Similarly, we can obtain an $\E$-triangle $\xymatrix{L_{1}\ar[r]^{x_{1}}&X\ar[r]^{y_{1}}&C_{1}\ar@{-->}[r]^{\delta_{1}}&.}$
Hence there is an $\E$-triangle $$\scalebox{0.9}[0.9]{\xymatrix@C=3.5em{L_{1}\oplus L_{2}\ar[r]^{\tiny\begin{bmatrix}x_{1}&0\\0&x_{2}\end{bmatrix}}&X\oplus X\ar[r]^{\tiny\begin{bmatrix}y_{1}&0\\0&y_{2}\end{bmatrix}}&C_{1}\oplus C_{2}\ar@{-->}[r]^{\delta_{1}\oplus\delta_{2}}&.}}$$ By Lemma \ref{lem1}, $L_{1}\oplus L_{2}\in \widehat{\X}$, and Proposition \ref{thm3} shows that resdim$_{\X}(L_{1}\oplus L_{2})\leq n-1$. By the induction hypothesis, $L_{1}$ and $L_{2}$ are in $\widehat{\X}$.
 It follows that $C_{1}$ and $C_{2}$ are in $\widehat{\X}$. Hence $\widehat{\X}$ is closed under direct summands. Hence $\mathrm{Thick(\X)}=\widehat{\X}$.
\end{proof}

\begin{definition} \emph{\cite[Definition 4.1]{NP}}\label{df:cotorsion pair}
{\rm Let $\mathcal{U}$, $\mathcal{V}$ $\subseteq$ $\mathcal{C}$ be a pair of full additive subcategories, closed
under isomorphisms and direct summands. The pair ($\mathcal{U}$, $\mathcal{V}$) is called a {\it cotorsion
pair} on $\mathcal{C}$ if it satisfies the following conditions:

(1) $\mathbb{E}(\mathcal{U}, \mathcal{V})=0$;

(2) For any $C \in{\mathcal{C}}$, there exists a conflation $V^{C}\rightarrow U^{C}\rightarrow C$ satisfying
$U^{C}\in{\mathcal{U}}$ and $V^{C}\in{\mathcal{V}}$;

(3) For any $C \in{\mathcal{C}}$ , there exists a conflation $C\rightarrow V_{C} \rightarrow U_{C}$ satisfying
$U_{C}\in{\mathcal{U}}$ and $V_{C}\in{\mathcal{V}}$.
}
\end{definition}

\begin{lem}\label{prop1}
Let $\X$ and $\W$ be two subcategories of $\C$ such that $\W$ is $\X$-injecive. Then the following statements hold.

(1) If $\W$ is a cogenerator for $\X$, then $\W=\X\bigcap\X^{\perp}=\X\bigcap\widehat{\W}$.

(2) If $\W$ is a cogenerator for $\X$, then $\widehat{\W}=\widehat{\X}\bigcap\X^{\perp}$.
\end{lem}
\begin{proof}
The proof is dual to that of \cite[Proposition 4.2]{MDZH}.
\end{proof}

The following result gives a method to construct cotorsion pairs on extriangulated categories.
\begin{prop}\label{prop4}
Let $(\X,\W)$ be a left Frobenius pair. Then $(\X, \widehat{\W})$ is a cotorsion pair on the extriangulated category $\mathrm{Thick}(\mathcal{X})$.
\end{prop}
\begin{proof}
  Note that $\mathrm{Thick}(\mathcal{X})$ is an extriangulated category by \cite[Remark 2.18]{NP}. It suffices to show $\widehat{\W}$ is closed under direct summands by Theorem \ref{thm2}. Note that  $\widehat{\W}=\widehat{\X}\bigcap\X^{\perp}$ by Proposition \ref{prop1}. Since $\mathrm{Thick}(\mathcal{X})$=$\widehat{\X}$ is closed under direct summands by Proposition \ref{cor2}, so is $\widehat{\W}$. This completes the proof.
\end{proof}
Now we are in a position to state and prove the main result of this section.
\begin{thm}\label{thm}
Let $\C$ be an extriangulated category. The assignments
\begin{center}
$(\X, \W)\mapsto (\X, \widehat{\W})$ $~~$ and $~~$ $(\mathcal{U},\mathcal{V})\mapsto ( \mathcal{U}, \mathcal{U}\bigcap\mathcal{V})$
\end{center}
 give mutually inverse bijections between the following classes:

(1) Left Frobenius pairs $(\X, \W)$.

(2) Cotorsion pairs $(\mathcal{U}, \mathcal{V})$ on the extriangulated category $\mathrm{Thick}(\mathcal{U})$ with $\mathcal{V}\subseteq\mathcal{U}^{\perp}$.
\end{thm}
\begin{proof}
 Let $(\X,\W)$ is a left Frobenius pair. Then $(\X, \widehat{\W})$ is a  cotorsion pair on the extriangulated category $\widehat{\X}$ by Proposition \ref{prop4}. Note that $\mathrm{Thick(\X)}=\widehat{\X}$ and $\widehat{\mathcal{W}}\subseteq\mathcal{X}^{\perp}$. Then
$(\X, \widehat{\W})$ is a  cotorsion pair on the extriangulated category $\mathrm{Thick}(\mathcal{X})$.

 Assume $(\mathcal{U}, \mathcal{V})$ is a cotorsion pair on the extriangulated category $\mathrm{Thick}(\mathcal{U})$ with $\mathcal{V}\subseteq\mathcal{U}^{\perp}$.
 For $U\in \mathcal{U}$, we have an $\E$-triangle
$\xymatrix{U\ar[r]^{x}&V\ar[r]^{y}&U'\ar@{-->}[r]^{\delta}&}$ with $V\in \mathcal{V}$ and $U'\in \mathcal{U}$.
 Hence $V\in\mathcal{U}\bigcap\mathcal{V}$. Note that $\mathcal{V}\subseteq\mathcal{U}^{\perp}$.
 It follows that $\mathcal{U}\bigcap\mathcal{V}$ is an $\mathcal{U}$-injective cogenerator.
 Let $\xymatrix{Z\ar[r]^{a}&U_{1}\ar[r]^{b}&U_{2}\ar@{-->}[r]^{\delta}&}$ be an $\E$-triangle with $U_{1}, U_{2}\in \mathcal{U}$.
 Then we have an exact sequence $\E(U_{1},V)\longrightarrow\E(Z,V)\longrightarrow\E^{2}(U_{2},V)$ for any $V\in \mathcal{V}$.
 Since $\mathcal{V}\subseteq\mathcal{U}^{\perp}$, $\E(Z,V)=0$.
 Note that $Z\in \mathrm{Thick}(\mathcal{U})$.
 Thus there exists an $\E$-triangle $\xymatrix{V^{Z}\ar[r]^{}& U^{Z}\ar[r]^{}& Z\ar@{-->}[r]^{}&}$ with $U\in \mathcal{U}$ and $V\in \mathcal{V}$ as $(\mathcal{U}, \mathcal{V})$ is a cotorsion pair on the extriangulated category $\mathrm{Thick}(\mathcal{U})$.
 Thus the above $\E$-triangle splits by $\E(Z,\mathcal{V})=0$.
 Hence $Z\in \mathcal{U}$. So $\mathcal{U}$ is closed under CoCone of deflations. Similarly, we can show that $\mathcal{U}$ is closed under extensions.
 Thus $( \mathcal{U}, \mathcal{U}\bigcap\mathcal{V})$ is a left Frobenius pair.

Based the above argument, it is enough to check that the compositions
\begin{center}
$(\mathcal{U}, \mathcal{V})\mapsto( \mathcal{U}, \mathcal{U}\bigcap\mathcal{V})\mapsto( \mathcal{U}, \widehat{\mathcal{U}\bigcap\mathcal{V}})$ $~~$ and $~~$ $(\X, \W)\mapsto (\X, \widehat{\W})\mapsto \X$
\end{center}
 are identities. Since $\mathcal{U}\bigcap\mathcal{V}$ is an $\mathcal{U}$-injective cogenerator for $\mathcal{U}$, $\widehat{\mathcal{U}\bigcap\mathcal{V}}=\widehat{\mathcal{U}}\bigcap \mathcal{U}^{\bot}=\mathrm{Thick}(\mathcal{U})\bigcap\mathcal{U}^{\bot}=\mathcal{V}$ where the first equality is due to Proposition \ref{prop1} and the second equality is due to Proposition \ref{cor2}.
\end{proof}

As a special case of Theorem \ref{thm}, we have the following result.
\begin{cor} \cite[Throrem 5.4]{BMPS}
Let $\mathcal{A}$ be an abelian category. The assignments
\begin{center}
$(\X, \W)\mapsto (\X, \widehat{\W})$ $~~$ and $~~$ $(\mathcal{U},\mathcal{V})\mapsto ( \mathcal{U}, \mathcal{U}\bigcap\mathcal{V})$
\end{center}
 give mutually inverse bijections between the following classes:

(1) Left Frobenius pairs $(\X, \W)$.

(2) Cotorsion pairs $(\mathcal{U}, \mathcal{V})$ on the exact category $\mathrm{Thick}(\mathcal{U})$ with $\mathcal{V}\subseteq\mathcal{U}^{\perp}$.
\end{cor}

As an application, we have the following result in \cite{MSSS2}.

\begin{cor}\cite[Theorem 3.11]{MSSS2} Let $\C$ be a triangulated category. The assignments
\begin{center}
$(\X, \W)\mapsto (\X, \widehat{\W})$ $~~$ and $~~$ $(\mathcal{U},\mathcal{V})\mapsto ( \mathcal{U}, \mathcal{U}\bigcap\mathcal{V})$
\end{center}
 give mutually inverse bijections between the following classes:

(1) Left Frobenius pairs $(\X, \W)$.

(2) Co-t-structures $(\mathcal{U}, \mathcal{V})$ on the triangulated category $\mathrm{Thick}(\mathcal{U}).$
\end{cor}
\begin{proof}
  Let $(\X,\W)$ be a left Frobenius pair. By Theorem \ref{thm}, $(\X, \widehat{\W})$ is a cotorsion pair on the triangulated category $\mathrm{Thick}(\mathcal{\X})$. Since $\X$ is closed under CoCone of deflations and extensions, it is easy to see that $\X[-1]\subseteq \X$. Hence $(\X, \widehat{\W})$ is a co-t-structure on the triangulated category $\mathrm{Thick}(\mathcal{\X})$.

  Assume $(\mathcal{U}, \mathcal{V})$ is a co-t-structure on the triangulated category $\mathrm{Thick}(\mathcal{U})$. It is easy to see that $(\mathcal{U}, \mathcal{V})$ is a cotorsion pair on $\mathrm{Thick}(\mathcal{U})$ with id$_{\mathcal{U}}(\mathcal{V})=0$. Hence the corollary follows from Theorem \ref{thm}.
\end{proof}

\begin{df} \cite[Definition 2.1]{HW}
Let $R$ and $S$ be rings. An $(S$-$R)$-bimodule $C={_{S}C}_{R}$ is semidualizing if:

(1) $_{S}C$ admits a degreewise finite $S$-projective resolution.

(2) $C_{R}$ admits a degreewise finite $R$-projective resolution.

(3) The homothety map $_{S}S_{S}\xrightarrow{_{S}\gamma}\mathrm{Hom}_{R}(C,C)$ is an isomorphism.

(4) The homothety map $_{R}R_{R}\xrightarrow{\gamma_{R}}\mathrm{Hom}_{S}(C,C)$ is an isomorphism.

(5) $\mathrm{Ext}_{S}^{\geq 1}(C,C)=0=\mathrm{Ext}_{R}^{\geq 1}(C,C)$.
\end{df}

\begin{df}\cite[Definition 3.1]{HW}
A semidualizing bimodule $C={_{S}C}_{R}$ is faithfully semidualizing if it satisfies the following conditions for all modules $_{S}N$ and $M_{R}$.

(1) If $\mathrm{Hom}_{S}(C,N)=0,$ then $N=0$.

(2) If $\mathrm{Hom}_{R}(C,M)=0,$ then $M=0$.
\end{df}

\begin{df}\cite[Definition 4.1]{HW}
The Bass class $\mathcal{B}_{C}(S)$ with respect to $C$ consists of all $S$-modules $N$ satisfying

(1) $\mathrm{Ext}_{S}^{\geq 1}(C,N)=0=\mathrm{Tor}^{R}_{\geq 1}(C,\mathrm{Hom}_{S}(C,N))=0$.

(2) The natural evaluation homomorphism $\nu_{N}:C\otimes_{R}\mathrm{Hom}_{S}(C,N)\rightarrow N$ is an isomorphism.
\end{df}

\begin{rem}
Let $C={_{S}C}_{R}$ be a faithfully semidualizing module. Then Bass class $\mathcal{B}_{C}(S)$ is an exact category by \cite[Theorem 6.2]{HW} and $\mathcal{B}_{C}(S)$ has enough projectives and injectives by \cite[Remark 3.13]{GD}.
\end{rem}

By \cite{HW}, the class of \emph{$C$-projective} left $S$-modules, denoted
by $\mathcal{P}_{C}(S)$ consists of those
left $S$-modules of the form $C\otimes_{R} P$ for some projective left $R$-module $P$. Recall from \cite{GD} that a left $S$-module $M$ is called \emph{$C$-Gorenstein projective} if there is an exact sequence of left $S$-modules
$$\mathbf{W}= \cdots\rightarrow W_{1}\rightarrow W_{0}\rightarrow W^{0}\rightarrow P^{1}\rightarrow\cdots$$
with each term in $\mathcal{P}_{C}(S)$ such that $N\cong\mathrm{Ker}(W_{0}\rightarrow W^{0})$ and both $\mathrm{Hom}_{R}(\mathbf{W},Q)$ and $\mathrm{Hom}_{R}(Q,\mathbf{W})$ are exact for any object $Q$ in $\mathcal{P}_{C}(S)$. It should be noted that $C$-Gorenstein projectives defined here
are different from those defined in \cite{HJs} when $S = R$ is a commutative Noetherian ring (see \cite[Proposition 3.6]{GD}).

For convenience, we write $G_{C}$-$Proj(S)$ for the classes of
$C$-Gorenstein projective left $R$-modules. By \cite[Proposition 3.5]{GD}, one has that $G_{C}$-$Proj(S)$ $\subseteq\mathcal{B}_{C}(S)$.
As a consequence of Theorem \ref{thm}, we have the following result.

\begin{cor} \label{cor1}
Let $C={_{S}C}_{R}$ be a faithfully semidualizing module. Then

(1) $(G_{C}$-$Proj(S),\mathcal{P}_{C}(S))$ is a strong left Frobenius pair in $\mathcal{B}_{C}(S)$.

(2) $(G_{C}$-$Proj(S),\widehat{\mathcal{P}_{C}(S)})$ is a cotorsion pair on $\widehat{G_{C}\text{-}Proj(S)}$.
\end{cor}
\begin{proof}
Since $\mathcal{P}_{C}(S)$ is projectively resolving and $\mathcal{P}_{C}(S)\bot \mathcal{P}_{C}(S)$ by \cite[Corollary 6.4]{HW} and  \cite[Theorem 6.4]{HW}, $G_{C}$-$Proj(S)$ is closed under kernels of epimorphisms and direct summand by \cite[Theorem 4.12]{SSW} and \cite[Proposition 4.11]{SSW}.
 Hence $(G_{C}$-$Proj(S),\mathcal{P}_{C}(S))$ is a strong left Frobenius pair in $\mathcal{B}_{C}(S)$. (2) follows from Theorem \ref{thm}.
\end{proof}

\section{\bf Admissible model structures associated with Frobenius pairs}

In this section, we shall use our results in Section 3 to construct more admissible model structures in extriangulated categories. At first, we need recall the following definition.

%
\begin{definition} \emph{ \cite[Definition 4.2]{NP}}\label{df:cotorsion pair}
{\rm Let $\mathcal{X}$, $\mathcal{Y}$  be two subcategories of $\C$. Define full subcategories $\textrm{Cone}(\mathcal{X},\mathcal{Y})$ and $\textrm{CoCone}(\mathcal{X},\mathcal{Y})$ of
$\mathcal{C}$ as follows.

(1) $C$ belongs to $\textrm{Cone}(\mathcal{X},\mathcal{Y})$ if and only if it admits a conflation
$X \rightarrow Y \rightarrow C$ satisfying $X\in{\mathcal{X}}$ and $Y\in{\mathcal{Y}}$;

(2) $C$ belongs to $\textrm{CoCone}(\mathcal{X},\mathcal{Y})$  if and only if it admits a conflation
$C\rightarrow X \rightarrow Y$ satisfying $X\in{\mathcal{X}}$ and $Y\in{\mathcal{Y}}$.}
\end{definition}

\begin{definition} \emph{ \cite[Definition 5.1]{NP}}\label{df:hovey-cotorsion pair}
{\rm Let ($\mathcal{S}$, $\mathcal{T}$) and ($\mathcal{U}$, $\mathcal{V}$) be cotorsion pairs on $\mathcal{C}$. Then $\mathcal{P}=((\mathcal{S}$, $\mathcal{T}$), ($\mathcal{U}$, $\mathcal{V}$)) is called a {\it twin cotorsion pair} if it satisfies
$\mathbb{E}(\mathcal{S}, \mathcal{V})=0$. Moreover, $\mathcal{P}$ is called a {\it Hovey twin cotorsion pair} if it satisfies
$\textrm{Cone}(\mathcal{V},\mathcal{S})$ = $\textrm{CoCone}(\mathcal{V},\mathcal{S})$.}
\end{definition}

In \cite{NP} Nakaoka and Palu gave a correspondence between admissible model structures and Hovey twin cotorsion pairs on $\mathcal{C}$. Essentially, an  admissible model structure on $\mathcal{C}$ is a Hovey twin cotorsion pair $\mathcal{P}=((\mathcal{S}$, $\mathcal{T}$), ($\mathcal{U}$, $\mathcal{V}$)) on $\mathcal{C}$. For more details, we refer to \cite[Section 5]{NP}.
 By a slight abuse of language we often refer to a Hovey twin cotorsion pair as an admissible model structure.

\begin{lem}\label{lem2}
Let $(\X,\W)$ be a strong left Frobenius pair in $\C$. Then $(\W, \widehat{\X})$ is a cotorsion pair on the extriangulated category $\mathrm{Thick}(\X)$.
\end{lem}
\begin{proof}
 Since $\W\subseteq{^{\bot}\X}$, one has $\E(\W,\widehat{\X})=0$.
For any $C\in \widehat{\X}$, there exists an $\E$-triangle
$$\xymatrix@C=0.6cm{Y_{C}\ar[r]^{}&X_{C}\ar[r]^{}&C\ar@{-->}[r]^{}&}$$ with $X_{C}\in \X$ and $Y_{C}\in \widehat{\W}$ by Theorem \ref{thm2}. Since $\W$ is a generator for $\X$, we have an $\E$-triangle
$\xymatrix@C=0.6cm{X\ar[r]^{}&W\ar[r]^{}&X_{C}\ar@{-->}[r]^{}&}$ with $W\in\W$ and $X\in \X$.
By $(ET4)^{op}$, we obtain a commutative diagram

$$\scalebox{0.9}[0.9]{\xymatrix{X\ar[r]^{}\ar@{=}[d]&Z\ar[r]^{}\ar[d]_{}&Y_{C}\ar[d]^{}&\\
X\ar[r]^{}&W\ar[r]^{}\ar[d]_{}&X_{C}\ar[d]^{}&\\
&C\ar@{=}[r]&C.}}$$
It follows that $Z\in \widehat{\X}$ as $X$ and $Y_{C}$ are in $\widehat{\X}$. Note that $\mathrm{Thick}(\X)=\widehat{\X}$. The second column and $\E(\W,\widehat{\X})=0$ show that $(\W,\widehat{\X})$ is a cotorsion pair on $\mathrm{Thick}(\X)$.
\end{proof}

\begin{prop}\label{prop8}
Let $(\X,\W)$ be a strong left Frobenius pair in $\C$. Then $\mathcal{P}=((\W, \widehat{\X}),(\X,\widehat{\W}))$ is an admissible model structure on the extriangulated category $\mathrm{Thick}(\X)$
\end{prop}
\begin{proof}
By Theorem \ref{thm} and Lemma \ref{lem2}, we only need to check that $\mathrm{Cone}(\widehat{\W},\W)=\mathrm{CoCone}(\widehat{\W},\W)$. It is obvious that $\mathrm{Cone}(\widehat{\W},\W)=\widehat{\W}\subseteq\mathrm{CoCone}(\widehat{\W},\W).$
Let $C\in \mathrm{CoCone}(\widehat{\W},\W)$. Then we have an $\E$-triangle
$\xymatrix@C=0.6cm{C\ar[r]^{}&Y\ar[r]^{}&W\ar@{-->}[r]^{}&}$ with $Y\in\widehat{\W}$ and $W\in \W$. By Theorem \ref{thm2}, one has that $C\in \widehat{\X}.$ Since $\E(\W,\widehat{\X})=0$, it follows that $C$ is a direct summand of $Y$. Note that $\widehat{\W}$ is closed under direct summand by Proposition \ref{prop4}.
 Thus $C\in \widehat{\W}$. Hence the equality $\mathrm{Cone}(\widehat{\W},\W)=\mathrm{CoCone}(\widehat{\W},\W)$ holds.
\end{proof}

\begin{thm}\label{thm6}
Let $(\X,\W)$ be a strong left Frobenius pair in $\C$. If $n$ is a non-negative integer, then the following statements are equivalent:

(1) $\widehat{\X_{n}}=\C$.

(2) $\mathcal{P}=((\W, \C),(\X,\widehat{\W_{n}}))$ is an admissible model structure on $\C$.
\end{thm}
\begin{proof}
$(1)\Rightarrow(2)$. If $\widehat{\X_{n}}=\C$, then $(\W,\C)$ is a cotorsion pair on $\C$ by Lemma \ref{lem2} and $(\X, \widehat{\W})$ is a cotorsion pair on $\C$ by Theorem \ref{thm}. To prove $(2)$, we only need to check that $\widehat{\W_{n}}=\widehat{\W}.$
Note that $\widehat{\W_{n}}\subseteq\widehat{\W}$ is obvious. Let $C\in \widehat{\W}$. Then there is an $\E$-triangle
$\xymatrix@C=0.6cm{Y_{C}\ar[r]^{}&X_{C}\ar[r]^{}&C\ar@{-->}[r]^{}&}$ with $X_{C}\in\X$ and $Y_{C}\in \widehat{\W_{n-1}}$ by Theorem \ref{thm2}.
Since $\widehat{\W}$ is closed under extensions, $X_{C}\in \X\bigcap \widehat{\W}=\W$. Hence $C\in \widehat{\W_{n}}$ implies $\widehat{\W_{n}}=\widehat{\W}$.

$(2)\Rightarrow(1)$. Since $(\X, \widehat{\W_{n}})$ is a cotorsion pair on $\C$, one has that $\C=\widehat{\X_{n}}$ by Theorem \ref{thm2}.
\end{proof}

As an application, we have the following result in \cite{HCc}.
\begin{cor}\cite[Theorem 8.6]{HCc}
 Suppose $R$ is a Gorenstein ring. Let $\mathcal{GP}(R)$ be the  subcategory of $R$-$\mathrm{Mod}$ consisting of Gorenstein projective modules and $\mathcal{P}(R)$ the subcategory of $R$-$\mathrm{Mod}$ consisting of projective modules. Then $((\mathcal{P}(R), R$-$\mathrm{Mod}),(\mathcal{GP}(R),\widehat{\mathcal{P}(R)})$ is an admissible model structure on $R$-$\mathrm{Mod}$.
\end{cor}
\begin{proof}
It follows from that Example \ref{ex} and Proposition \ref{prop8}.
\end{proof}

Let $n$ be a non-negative integer. In the following, we denote by $G_{C}\text{-}Proj(S)_{\leq n}$ (resp.,~$\mathcal{P}_{C}(S)_{\leq n})$ the class of modules with $C$-Gorenstein projective (resp., $C$-projective) dimension at most $n$
\begin{cor}
Let $C={_{S}C}_{R}$ be a faithfully semidualizing module. Then the following statements are equivalent:

(1) $G_{C}\text{-}Proj(S)_{\leq n}=\mathcal{B}_{C}(S)$.

(2) $\mathcal{P}=((\mathcal{P}_{C}(S),\mathcal{B}_{C}(S)),(G_{C}\text{-}Proj(S),\mathcal{P}_{C}(S)_{\leq n})$ is an admissible model structure on $\mathcal{B}_{C}(S)$.
\end{cor}
\begin{proof}
It follows from Corollary \ref{cor1} and Theorem \ref{thm6}.
\end{proof}

Let  $(\mathcal{C}, \mathbb{E}, \mathfrak{s})$ be an extriangulated category and $\xi$ a proper class of $\mathbb{E}$-triangles. By \cite{HZZ}, an object $P\in\mathcal{C}$  is called {\it $\xi$-projective}  if for any $\mathbb{E}$-triangle $$\xymatrix{A\ar[r]^x& B\ar[r]^y& C \ar@{-->}[r]^{\delta}& }$$ in $\xi$, the induced sequence of abelian groups $\xymatrix@C=0.6cm{0\ar[r]& \mathcal{C}(P,A)\ar[r]& \mathcal{C}(P,B)\ar[r]&\mathcal{C}(P,C)\ar[r]& 0}$ is exact. We denote $\mathcal{P(\xi)}$ the class of $\xi$-projective objects of $\mathcal{C}$. Recall from \cite{HZZ} that an object $M\in\mathcal{C}$ is called \emph{$\xi$-$\mathcal{G}$projective} if there exists a diagram
$$\xymatrix@C=2em{\mathbf{P}:\cdots\ar[r]&P_1\ar[r]^{d_1}&P_0\ar[r]^{d_0}&P_{-1}\ar[r]&\cdots}$$ in $\mathcal{C}$ satisfying that: (1) $P_n$ is $\xi$-projective for each integer $n$; (2) there is a $\mathcal{C}(-,\mathcal{P(\xi)})$-exact $\mathbb{E}$-triangle $\xymatrix@C=2em{K_{n+1}\ar[r]^{g_n}&X_n\ar[r]^{f_n}&K_n\ar@{-->}[r]^{\delta_n}&}$ in $\xi$ and $d_n=g_{n-1}f_n$ for each integer $n$ such that $M\cong K_n$ for some $n\in{\mathbb{Z}}$. We denote  by $\mathcal{GP}(\xi)$ the class of $\xi$-$\mathcal{G}$projective objects in $\mathcal{C}$.
Specializing Theorem \ref{thm6} to the case $\mathcal{X}=\mathcal{GP}(\xi)$, we have the following result in \cite{HZZ}.

\begin{cor}\cite[Theorem 5.9]{HZZ}\label{cor:th} Let  $(\mathcal{C}, \mathbb{E}, \mathfrak{s})$ be an extriangulated category satisfying Condition \emph{(WIC)} {\rm(}see \cite[Condition 5.8]{NP}{\rm)}. Assume that $\xi$ is a proper class in $\mathcal{C}$.
Set $\mathbb{E}_\xi:=\mathbb{E}|_\xi$, that is, $$\mathbb{E}_\xi(C, A)=\{\delta\in\mathbb{E}(C, A)~|~\delta~ \textrm{is realized as an $\mathbb{E}$-triangle}\xymatrix{A\ar[r]^x&B\ar[r]^y&C\ar@{-->}[r]^{\delta}&}~\textrm{in}~\xi\}$$ for any $A, C\in\mathcal{C}$, and $\mathfrak{s}_\xi:=\mathfrak{s}|_{\mathbb{E}_\xi}$.
If $n$ is a non-negative integer, then the following conditions are equivalent:

\emph{(1)} ${\rm sup}\{\xi$-$\mathcal{G}{\rm pd}A|A\in\mathcal{C}\}\leqslant n$.

\emph{(2)} $\mathcal{P}=((\mathcal{P}(\xi),\mathcal{C}), (\mathcal{GP}(\xi), \mathcal{P}^{\leqslant n}(\xi)))$ is an admissible model structure on $(\mathcal{C}, \mathbb{E}_\xi, \mathfrak{s}_\xi)$, where $\mathcal{P}^{\leqslant n}(\xi)=\{A\in\mathcal{C}|\xi$-${\rm pd}A\leqslant n\}$.
\end{cor}
\begin{proof}
It is easy to check that $(\mathcal{GP}(\xi),\mathcal{P}(\xi))$ is a strong left Frobenius pair. Thus the corollary follows from  Theorem \ref{thm6}.
\end{proof}

\renewcommand\refname
\textbf{Yajun Ma}\\
Department of Mathematics, Nanjing University, Nanjing 210093, China.\\
E-mail: \textsf{13919042158@163.com}\\[1mm]
\textbf{Haiyu Liu}\\
School of Mathematics and Physics, Jiangsu University of Technology,
 Changzhou, Jiangsu 213001, China.\\
E-mail: \textsf{haiyuliumath@126.com}\\[1mm]
\textbf{Yuxian Geng}\\
School of Mathematics and Physics, Jiangsu University of Technology,
 Changzhou, Jiangsu 213001, China.\\
E-mail: \textsf{yxgeng@jsut.edu.cn}\\[1mm]
\end{document}